\documentclass[11pt]{article}
\usepackage{amsmath,amsthm,amssymb,color,verbatim,graphicx,fullpage,url}
\newcommand{\remove}[1]{}
\newtheorem{theorem}{Theorem}[section]
\newtheorem{claim}[theorem]{Claim}
\newtheorem{lemma}[theorem]{Lemma}
\newtheorem{definition}[theorem]{Definition}
\newtheorem{corollary}[theorem]{Corollary}
\newtheorem{example}[theorem]{Example}
\newtheorem{conjecture}[theorem]{Conjecture}

\newcommand{\eps}{\varepsilon}
\newcommand{\sprod}[2]{\langle{#1},{#2}\rangle}

\newcommand{\spec}[2]{\text{Spec}_{#1}({#2})}

\newcommand{\E}{\mathbb{E}}
\newcommand{\C}{\mathbb{C}}
\newcommand{\Z}{\mathbb{Z}}

\title{On the structure of the spectrum of small sets}

\author{
Kaave Hosseini
\thanks{Supported by NSF CAREER award 1350481.}
\\
Computer Science and Engineering\\
University of California, San Diego\\
\texttt{skhossei@ucsd.edu}
\and
Shachar Lovett
\thanks{Supported by NSF CAREER award 1350481.}\\
Computer Science and Engineering\\
University of California, San Diego\\
\texttt{slovett@ucsd.edu}
}

\begin{document}
\maketitle

\begin{abstract}
Let $G$ be a finite abelian group and $A$ a subset of $G$. The spectrum of $A$ is the set of its large Fourier coefficients.
Known combinatorial results on the structure of spectrum, such as Chang's theorem, become trivial in the regime $|A| = |G|^\alpha $ whenever $\alpha \le c$, where $c \ge 1/2$ is some absolute constant. On the other hand, there are statistical results, which apply only to a noticeable fraction of the elements, which give nontrivial bounds even to much smaller sets.
One such theorem (due to Bourgain) goes as follows. For a noticeable fraction of pairs $\gamma_1,\gamma_2 $ in the spectrum, $\gamma_1+\gamma_2$
 belongs to the spectrum of the same set with a smaller threshold.
Here we show that this result can be made combinatorial by restricting to a large subset. That is, we show that for any set $A$ there exists a large subset $A'$,
such that the sumset of the spectrum of $A'$ has bounded size. Our results apply to sets of size $|A| = |G|^{\alpha}$ for any constant $\alpha>0$, and even in some sub-constant regime.
\end{abstract}

\section{Introduction}

Let $G$ be a finite Abelian group, and let $A$ be a subset of $G$. For a character $\gamma \in \widehat{G}$, the corresponding Fourier coefficient of $1_A$ is
$$
\widehat{1_A}(\gamma) = \sum_{x\in A} \gamma(x).
$$
The spectrum of $A$ is the set of characters with large Fourier coefficients,
$$
\spec{\eps}{A} = \{\gamma\in \widehat{G}: |\widehat{1_A}(\gamma)|\geq \eps |A| \}.
$$
Note that the spectrum of a set is a symmetric set, that is $\spec{\eps}{A} = - \spec{\eps}{A}$, where we view $\widehat{G}$ as an additive group (which is isomorphic to $G$).

Understanding the structure of the spectrum of sets is an important topic in additive combinatorics, with several striking applications discussed below. As we illustrate, there is a gap in our knowledge between \emph{combinatorial} structural results, which apply to all elements in the spectrum, and \emph{statistical} structural results, which apply to most elements in the spectrum. The former results apply only to large sets, typically of the size $|A| \ge |G|^c$ for some absolute constant $c>0$, where the latter results apply also for smaller sets. The goal of this paper is to bridge this gap.

Our interest in this problem originates from applications of it in computational complexity, where a better understanding of the structure of the spectrum of small sets can help to shed light on some of the main open problems in the area, such as constructions of two source extractors~\cite{bourgain2005more,rao2007exposition,BenSassonZewi11} or the log rank conjecture in communication complexity~\cite{BenSassonLovettZewi11}. We refer the interested reader to a survey by the second author on applications of additive combinatorics in theoretical computer science~\cite{lovett2014survey}. In this paper we focus on the core mathematical problem, and do not discuss applications further.

We assume from now on that $|A| = |G|^{\alpha}$ where $\alpha>0,\eps>0$ are arbitrarily small constants, which is the regime where current techniques fail. In fact, 
our results extend to some range of sub-constant parameters, but only mildly. First, we review the current results on the structure of the spectrum, and their limitations.

\paragraph{Size bound.}
The most basic property of the spectrum is that it cannot be too large.
Parseval's identity bounds the size of the spectrum by
$$
|\spec{\eps}{A}|\leq \frac{|G|}{\eps^2 |A|} = \frac{|G|^{1-\alpha}}{\eps^2}.
$$
However, this does not reveal any information about the structure of the spectrum, except from a bound on its size.

\paragraph{Dimension bound.}
A combinatorial structural result on the spectrum was obtained by Chang~\cite{chang2002polynomial}. She discovered that the spectrum is low dimensional. For a set $\Gamma \subseteq \widehat{G}$, denote its \emph{dimension} as the minimal integer $d$, such that there exist $\gamma_1,\ldots,\gamma_d \in \widehat{G}$ with the following property: any element $\gamma \in \Gamma$ can be represented as $\gamma = \sum \eps_i \gamma_i$ with $\eps_i \in \{-1,0,1\}$. With this definition, Chang's theorem asserts that
$$
\dim(\spec{\eps}{A}) \le O(\eps^{-2} \log(|G|/|A|)).
$$
Chang~\cite{chang2002polynomial} used this result to obtain improved bounds for Freiman's theorem on sets with small doubling, and Green~\cite{green2002arithmetic} used it to find
arithmetic progressions in sumsets. Moreover, Green~\cite{greensumset} showed that the bound in Chang's theorem cannot in general be improved, at least when $A$ is not too small. Recently, Bloom~\cite{bloom} obtained sharper bounds for a large subset of the spectrum. He showed that there exists a subset $\Gamma \subseteq \spec{\eps}{A}$
of size $|\Gamma| \ge \eps \cdot |\spec{\eps}{A}|$ such that
$$
\dim(\Gamma) \le O(\eps^{-1} \log(|G|/|A|)).
$$
He applied these structural results to obtain improved bounds for Roth's theorem and related problems. However, we note that in our regime of interest, where $|A|=|G|^{\alpha}$ with $0<\alpha<1$, both results become trivial if $\eps$ is a small enough constant. This is because both give a bound on the dimension of the form $O(\eps^{-c} (1-\alpha)) \cdot \log |G|$ with $c \in \{1,2\}$. However, any set $\Gamma \subseteq \widehat{G}$ trivially has dimension at most $\log |G|$. As our interest is in the regime of any arbitrarily small constant $\alpha,\eps>0$, we need to turn to a different set of technqiues.

\paragraph{Statistical doubling.}
Bourgain~\cite{bourgainStat} showed that for many pairs of elements in the spectrum, their sum lands in a small set. Concretely,
$$
\Pr_{\gamma_1,\gamma_2 \in \spec{\eps}{A}}[\gamma_1 + \gamma_2 \in \spec{\eps^2/2}{A}] \ge \eps^2/2,
$$
where we note that by Parseval's identity, $|\spec{\eps^2/2}{A}| \le O(|G|^{1-\alpha}/\eps^4)$. He used these results to obtain improved bounds on exponential sums. Similar
bounds can be obtained for linear combinations of more than two elements in the spectrum, for example as done by Shkredov~\cite{shkredov}. If we assume that
$|\spec{\eps^2/2}{A}| \le K |\spec{\eps}{A}|$ and
apply the Balog-Szemer\'{e}di-Gowers theorem
\cite{balog1994statistical,gowers1998new}, this implies that there exists a large subset $\Gamma \subseteq \spec{\eps}{A}$ such that $|\Gamma+\Gamma| \le (K/\eps)^{O(1)} |\Gamma|$. However, it does not provide
any bounds on the sumset of the entire spectrum, that is on
$|\spec{\eps}{A}+\spec{\eps}{A}|$. In fact, we will later see an example showing that this sumset could be much large than the spectrum, whenever $\eps \le 1/2$.

\paragraph{Combinatorial doubling.}
The motivating question for the current work is to understand whether the statistical doubling result described above, can be applied for the entire spectrum. That is, can we obtain combinatorial structural results
on the sumset of the entire spectrum $\spec{\eps}{A}+\spec{\eps}{A}$.

As a first step, we ask for which $\alpha,\eps>0$ is is true that, for any set $A$ of size $|A|=|G|^{\alpha}$, the sumset $\spec{\eps}{A}+\spec{\eps}{A}$ is much smaller than the entire group.
There are two regimes where this is trivially true. First, when $\alpha>1/2$, it is true since by Parseval's identity, $\spec{\eps}{A}$ is smaller than the square root of the group size, and hence
$$
|\spec{\eps}{A} + \spec{\eps}{A}| \le |\spec{\eps}{A}|^2 \le \frac{|G|^{2-2\alpha}}{\eps^4} .
$$
Also, when $\eps>3/4$ then $\spec{\eps}{A}+\spec{\eps}{A} \subseteq \spec{4 \eps-3}{A}$ (see, e.g.,~\cite{tao2006additive} for a proof) and hence again by Parseval's identity, the size of the sumset is bounded by
$$
|\spec{\eps}{A} + \spec{\eps}{A}| \le |\spec{\eps}{A}|^2 \le \frac{|G|^{1-\alpha}}{(4\eps-3)^2} .
$$
As the following example shows, the thresholds of $\alpha=1/2, \eps=1/2$ are tight.

\begin{example}
Let $G = \Z_2^{2n}$ and $A = (\Z_2^n\times \{0^n\}) \cup  (\{0^n\}\times \Z_2^n)$. Then $|A|= 2 |G|^{1/2}-1$, $\spec{1/2}{A} = A$ and $A+A = G$.
\end{example}

So, it seems that such structural results are hopeless when $\alpha,\eps<1/2$. However, there is still hope: in the example, if we restrict to a large subset $A' = \Z_2^n\times\{0^n\}\subseteq A$, then
$\spec{1/2}{A^\prime} = \{0^n\} \times \Z_2^n$ is a subgroup, and specifically the size of $\spec{1/2}{A^\prime}+\spec{1/2}{A^\prime}$ is bounded away from the entire group. Our first result is that this is true in general. In fact, the size of the sumset is close to the bound given by Parseval's identity, which is approximately $|G|^{1-\alpha}$.

\begin{theorem}\label{thm:spec_dbl}
Fix $0<\delta<\alpha<1/2$ and $0<\eps<1/2$. Let $A\subseteq G$ of size $|A| \geq |G|^\alpha$. Then there exists a subset $A' \subseteq A$ of size
$|A'| \ge |A|/C$ such that
$$
\left| \spec{\eps}{A'} + \spec{\eps}{A'} \right| \le C  |G|^{1-\alpha+\delta}
$$
where $C \le \exp((1/\eps)^{O(1/\delta)})$.
\end{theorem}

A more refined notion of structure is that of bounded doubling. Here, we say that a set $\Gamma$ has a doubling constant $K$ if $|\Gamma+\Gamma| \le K |\Gamma|$.
Note that if $|\spec{\eps}{A'}|$ has size close to the bound given by Parseval's identity, which is roughly $|G|^{1-\alpha}$, then Theorem~\ref{thm:spec_dbl} would show that $\spec{\eps}{A'}$ has a small doubling constant $K=C |G|^{\delta}$. We conjecture that this is always the case. However, we could only show it if we are allowed to change the value of $\eps$ somewhat. We state both the theorem and the conjecture below.

\begin{theorem}\label{thm:spec_dbl2}
Fix $0<\delta<\alpha<1/2$ and $0<\eps<1/2$. Let $A\subseteq G$ of size $|A| \geq |G|^\alpha$.
Then there exists a subset $A' \subseteq A$ of size $|A'| \ge |A|/C$ and $\eps' \ge \eps^{2^{1/\delta}}$ such that
$$
|\spec{\eps'}{A'}| \ge |\spec{\eps}{A}| / C
$$
and
$$
|\spec{\eps'}{A'}+\spec{\eps'}{A'}| \le C |G|^{\delta} \cdot |\spec{\eps'}{A'}|,
$$
where $C \le \exp \left((1/\eps)^{O(2^{4/\delta})} \right)$.
\end{theorem}

\begin{conjecture}\label{conj:spec_dbl}
Fix $0<\delta<\alpha<1/2$ and $0<\eps<1/2$. Let $A\subseteq G$ of size $|A| \geq |G|^\alpha$.
Then there exists a subset $A' \subseteq A$ of size $|A'| \ge |A|/C$ such that
$$
|\spec{\eps}{A'}+\spec{\eps}{A'}| \le C |G|^{\delta} \cdot |\spec{\eps}{A'}|,
$$
where $C=C(\eps,\delta)$.
\end{conjecture}

\paragraph{Notations.}
We use big-O notation. For two quantities $x,y$, the expression $x = O(y)$ means $x \le cy$ for an unspecified absolute constant $c>0$. We also use $c,c',c_1$, etc to denote
unspecified absolute constants, where the big-O notation may be confusing. The value of these may change between different instantiations of them. We make no effort to
optimize constants.

\paragraph{Paper organization.}
We prove Theorem~\ref{thm:spec_dbl} in Section~\ref{sec:proof1} and Theorem~\ref{thm:spec_dbl2} in Section~\ref{sec:proof2}.

\section{Proof of Theorem~\ref{thm:spec_dbl}}
\label{sec:proof1}

We begin by introducing some notation. For $A \subseteq G$ and $\Gamma\subseteq \widehat{G}$, define an $|A| \times |\Gamma|$ complex matrix $ M = M(A, \Gamma)$, with rows
indexed by $A$ and columns by $\Gamma$, as follows.
First, denote by $\gamma(A) := \E_{a\in A} [\gamma(a)]$ the average value of the character $\gamma$ on $A$. Define
$$
M_{a,\gamma} := \gamma(a) \frac{\overline{\gamma(A)}}{|\gamma(A)|}.
$$
With this definition, we have that for any $\Gamma \subseteq \spec{\eps}{A}$,
\begin{equation}\label{equation:m}
\left| \mathbf{1}_A^T  M(A,\Gamma) \mathbf{1}_{\Gamma} \right|
= \sum_{\gamma \in \Gamma} \left| \sum_{a \in A} \gamma(a) \right|
\geq \eps |A| |\Gamma| .
\end{equation}
We next define a notion of regularity for $M(A,\Gamma)$.

\begin{definition}[Regularity for $M(A,\Gamma)$]
Let $A\subseteq G, \Gamma \subseteq \widehat{G}$. The matrix $M = M(A,\Gamma)$  is called $\lambda$-\textit{regular} if for every pair of functions $f:A \rightarrow \mathbb{C}$, $g:\Gamma \rightarrow \mathbb{C}$ such that $\sprod{f}{\mathbf{1}_A} = 0 $ or $\sprod{g}{\mathbf{1}_{\Gamma}} = 0 $ or both,  it holds that $$| f^T  M g| < \lambda \| f\|_{\infty} \| g\|_{\infty} |A||\Gamma| .$$
\end{definition}

It is conventional to use the $L_2$-norm in definition of regularity, however in our case, the use of $L_{\infty}$-norm makes the argument more straightforward and gives better bounds.

The argument informally goes as follows. We divide into two cases. First, we show if $M = M(A,\spec{\eps}{A})$ is $\lambda$-regular for a suitable choice of $\lambda$,
then $\spec{\eps}{A}$ has bounded doubling. Otherwise, if $M$ is not $\lambda$-regular, we find large subsets $A' \subseteq A, \Gamma' \subseteq \spec{\eps}{A}$ such that $M(A', \Gamma')$ has
higher average. This allows us to revert to study $M(A', \spec{\eps'}{A'})$ where $\eps'=\eps+\lambda^{O(1)}$ and iterate.

First, we analyze the case where $M$ is regular.

\begin{lemma}\label{lemma:regular}
Fix some $0<\eps, \rho<1$ and $\Gamma \subseteq \spec{\rho}{A}  $. If $M = M(A,\Gamma )$ is $\eps\rho/150$-regular, then for any $\gamma \in \spec{\eps}{A}$, there is a subset $\Gamma_\gamma \subseteq \Gamma$, $| \Gamma_\gamma | \ge 0.9 | \Gamma |$ such that
$$
\gamma + \Gamma_\gamma \in \spec{\eps \rho/2}{A}.
$$
\end{lemma}

\begin{proof}
Suppose towards contradiction that there is some $\gamma_\circ \in \spec{\eps}{A}$ for which the claim does not hold. That is, there exists a subset $\Gamma' \subseteq \Gamma$ of size
 $| \Gamma' | > 0.1 | \Gamma|$ such that $\forall \gamma' \in \Gamma'$,
$$
\gamma_\circ + \gamma' \notin \spec{\eps \rho/2}{A}.
$$
Define a pair of functions $f: A \rightarrow \mathbb{C}$ and $g: \Gamma \rightarrow \mathbb{C}$ by
\begin{eqnarray*}
f(a) &=& \gamma_\circ(a),\\
g(\gamma)  &=& \frac{|\Gamma|}{|\Gamma'|} \mathbf{1}_{\Gamma'}(\gamma).
\end{eqnarray*}
We have
\begin{align*}
f^T  M g = & \sum_{\gamma\in \Gamma} \left[ \sum_{a\in A} \gamma_\circ (a) \gamma(a) \frac{\overline{\gamma(A)}}{|\gamma(A)|}\right]  \left[ \frac{|\Gamma|}{|\Gamma'|}\mathbf{1}_{\Gamma'}(\gamma)\right] \\
& = \frac{|\Gamma|}{|\Gamma'|} \sum_{\gamma\in \Gamma} \frac{\overline{\gamma(A)}}{|\gamma(A)|} \sum_{a\in A} \gamma_\circ(a) \gamma(a) \mathbf{1}_{\Gamma'}(\gamma)\\
& = \frac{|\Gamma|}{|\Gamma'|} \sum_{\gamma' \in \Gamma'} \frac{\overline{\gamma'(A)}}{|\gamma'(A)|} \sum_{a \in A} \left( \gamma_\circ + \gamma'\right) (a) .
\end{align*}
By our assumption, $\forall \gamma' \in \Gamma', \gamma_\circ + \gamma' \notin \spec{\eps \rho/2}{A}$. Therefore
$$\left| f^T  Mg\right| \leq (\eps \rho/2) \cdot | \Gamma| | A | .$$
Decompose $f$ as $f = f_1 + f_2$ with $ f_1 = \E_{a\in A} [f(a)]\cdot \mathbf{1}_A $ and $g$ as $g = g_1 + g_2 $ with $ g_1 =  \E_{\gamma\in \Gamma} [g(\gamma)] \cdot \mathbf{1}_{\Gamma}
= \mathbf{1}_{\Gamma}$. Then

\begin{equation}\label{equation:eq}
f^T  M g = f_1^T  M g_1 + f_2^T  M g_1 + f_1^T  M g_2 +f_2^T  M g_2 .
\end{equation}
We have that $\langle f_2 , \mathbf{1}_A\rangle = 0 $, $\langle g_2,\mathbf{1}_\Gamma\rangle = 0 $ and
$$
\left| f_1^T  M g_1\right| = \left| \E_{a\in A}f(a) \cdot \left( \mathbf{1}_A^T  M \mathbf{1}_\Gamma \right)\right| \geq \left| \E_{a\in A} [\gamma_\circ(a)]\right| \cdot \rho | \Gamma| | A |  \ge \eps \rho | \Gamma| | A | .
$$
We show that the other terms in Equation \eqref{equation:eq} are too small to cancel out the contribution of $f_1^T  M g_1$. Consequently, we reach a contradiction.

In each one of the terms $f_1^T  M g_2,f_2^T  M g_1,f_2^T  M g_2$ at least one of the functions are orthogonal to the identity function. Therefore, we can bound the size of these terms using the $\frac{\eps\rho}{150}$-regularity assumption. We have $\|f_1\|_\infty \leq 1, \|f_2\|_\infty\leq 2, \|g_1\|_\infty \leq 10, \|g_2\|_\infty\leq 10$, and hence
$$
\left| f_2^T M g_1 + f_1^T M g_2 + f_2^T M g_2\right| \le (20+10+20) \cdot (\eps \rho/150)|A| |\Gamma| =
(\eps\rho/3) |A||\Gamma|.
$$
This implies that $\left| f^T  M g \right| \ge \frac{2}{3} \eps \rho |A| |\Gamma|$, which is a contradiction.
\end{proof}

Next, we show how to use Lemma~\ref{lemma:regular} to infer that if $M = M(A, \spec{\rho}{A})$ is  $\frac{\eps\rho}{150}$-regular then $\left| \spec{\eps}{A} - \spec{\eps}{A}\right|$ is small as long as $|\spec{\eps \rho/2}{A}| \approx |\spec{\rho}{A}|$.

\begin{lemma}\label{lemma:regdouble}
If $M = M(A, \spec{\rho}{A}) $ is $\frac{\eps\rho}{150}$-regular, then
$$\left| \spec{\eps}{A} - \spec{\eps}{A}\right| \leq 2 \frac{\left| \spec{\eps \rho/2}{A}\right|^2}{ \left| \spec{\rho}{A}\right|} .$$
\end{lemma}

\begin{proof}
Fix arbitrary $\gamma_1, \gamma_2 \in \spec{\eps}{A}$. By Lemma~\ref{lemma:regular} there exist
sets $\Gamma_1, \Gamma_2 \subseteq \spec{\rho}{A}$ of size $|\Gamma_1|, |\Gamma_2| \ge 0.9 |\spec{\rho}{A}|$ such that $\gamma_1 + \Gamma_1, \gamma_2 + \Gamma_2 \subseteq \spec{\eps \rho/2}{A}$.
For any $\gamma \in \Gamma_1 \cap \Gamma_2$ we can then write
$$
\gamma_1 - \gamma_2 = (\gamma_1 + \gamma) - (\gamma_2 + \gamma)
$$
where $\gamma_1 + \gamma , \gamma_2 + \gamma \in \spec{\eps \rho/2}{A}$. This gives $|\Gamma_1 \cap \Gamma_2| \ge 0.8 |\spec{\rho}{A}|$ distinct ways to write
$\gamma_1 - \gamma_2$ as the difference of a pair of elements in $\spec{\eps \rho/2}{A}$.
Consequently
\begin{align*}
\left| \spec{\eps}{A} - \spec{\eps}{A}\right|  \leq \frac{\left| \spec{\eps \rho/2}{A}\right|^2}{\left| \Gamma_1 \cap \Gamma_2\right|}  \leq \frac{\left| \spec{\eps \rho/2}{A}\right|^2}{0.8 \left| \spec{\rho}{A}\right|} .
\end{align*}
\end{proof}

Next, we consider the case that the matrix $M$ is not $\lambda$-regular for $\lambda = \eps \rho/150$. In the following we denote $\E[M] := \E_{a,\gamma} [M_{a,\gamma}]$.
\begin{lemma}\label{lemma:minor}
If $M = M (A, \Gamma)$ is not $\lambda$-regular, then there exist subsets $A' \subseteq A$, $\Gamma' \subseteq \Gamma$ such that
$$
\left| \E \left[ M(A', \Gamma')\right]\right| \geq \left| \E \left[  M(A, \Gamma)\right] \right| + c\lambda^{15},
$$
where $|A'| \geq c \lambda^{15} |A|, |\Gamma'| \geq c \lambda^{15} |\Gamma|$, and $c>0$ is an absolute constant.
\end{lemma}

Assuming that  $M  = M(A, \Gamma)$ is not $\lambda$-regular, there are functions $f, g$ with $\|f\|_{\infty}=\|g\|_{\infty}=1$, at least one of which is orthogonal to the identity function,
such that $\left| f^T  M g \right| \geq \lambda |A||\Gamma|$. As a first step towards proving Lemma~\ref{lemma:minor}, we approximate $f,g$ by step functions $\widetilde{f}$ and $\widetilde{g}$, respectively.

\begin{claim}\label{claim:step}
Fix $\eta>0$. Let $f:A \to \C$ be a function with $\|f\|_{\infty}=1$. Then there exists a function $\widetilde{f}:A \to \C$ such that
$$\| f - \widetilde{f}\|_\infty \leq \eta$$
with $\widetilde{f} = \sum_{i=1}^{k} \alpha_i \mathbf{1}_{A_i}$, where $A_i \subseteq A$ are disjoint subsets and $\alpha_i \in \mathbb{C}$ with $|\alpha_i| \le 1$. Moreover, $k \leq \frac{100 }{\eta^2}$.
\end{claim}

\begin{proof}
We partition $A$ based on the phase and magnitude of $f$. For $r=\lceil 10/\eta \rceil$ define
$$
A_{j, k} = \left\{ a\in A:  j/r < |f(a)| \leq (j+1)/r \textrm { and } 2\pi k/r < \arg{f(a)} \leq 2\pi (k+1)/r \right\}.
$$
We partition $A$ to subsets $A_{j, k}$ for $j,k \in \{0,\ldots,r-1\}$. Define the step function $\widetilde{f}$ as
$$
\widetilde{f} = \sum_{j,k=0}^{r-1} j/r \cdot e^{(2\pi i) k/r } \cdot \mathbf{1}_{A_{j,k}}.
$$
It is easy to verify that for all $a \in A$, $| f(a) - \widetilde{f} (a)| \leq \eta |f(a)|$ as claimed.
\end{proof}

We proceed with the proof of Lemma~\ref{lemma:minor}.

\begin{proof}[Proof of Lemma~\ref{lemma:minor}]
Let $\rho := \E[M]$ be the average of $M$, and define a matrix $M'$ by $M'_{a,\gamma} = M_{a,\gamma}-\rho$, so that $\E[M']=0$. Note that $|M'_{a,\gamma}| \le 2$
for all $a \in A, \gamma \in \Gamma$. We may assume for simplicity that $\rho$ is real
and nonnegative, by multiplying all entries of $M$ by an appropriate phase $e^{i \theta}$, as this does not change any of the properties at hand.

As we assume $M$ is not $\lambda$-regular, there exist functions $f:A \to \C, g: \Gamma \to \C$ with $\|f\|_{\infty},\|g\|_{\infty}=1$, one of which at least sums to zero, such that
$|f^T M g| \ge \lambda |A| |\Gamma|$. Note that $f^T M' g = f^T M g$. Let $\widetilde{f},\widetilde{g}$ be their step function approximations given by Claim~\ref{claim:step} for $\eta = \lambda/8$,
where $\widetilde{f} = \sum_{i=1}^{k} \alpha_i \mathbf{1}_{A_i}, \widetilde{g} = \sum_{i=1}^{k} \beta_i \mathbf{1}_{\Gamma_i}$ and $k\leq \frac{100}{\eta^2}$. Moreover
$$
\left| \widetilde{f}^T M' \widetilde{g}\right| \geq |f^T M' g| - |(f-\widetilde{f})^T M' g| - |\widetilde{f}^T M' (g-\widetilde{g})| \ge \lambda/2\cdot |A| |\Gamma|.
$$
That is,
$$
\left| \sum_{i,j=1}^k \alpha_i \beta_j \mathbf{1}_{A_i}^T M' \mathbf{1}_{\Gamma_j} \right| \ge \lambda/2 \cdot |A| |\Gamma|.
$$
In particular, there must exist $A_i, \Gamma_j$ such that
$$
\left| \mathbf{1}_{A_i}^T M' \mathbf{1}_{\Gamma_j} \right| \ge (\lambda/2k^2) \cdot |A| |\Gamma| \ge c_1 \lambda^5 \cdot |A| |\Gamma|,
$$
where $c_1 > 0$ is an absolute constant.

If we knew that $\mathbf{1}_{A_i}^T M' \mathbf{1}_{\Gamma_j}$ is real and nonnegative, say, then
we would be done by choosing $A'=A_i, \Gamma'=\Gamma_j$ as then $\E[M(A',\Gamma')] \ge \rho + c_1 \lambda^5$. However, it may be that its real part is negative, cancelling the average.
To overcome this, we consider choosing $A' \in \{A_i, A_i^c\}, \Gamma' \in \{\Gamma_j, \Gamma_j^c\}$ (where $A_i^c = A \setminus A_i, \Gamma_j^c = \Gamma \setminus \Gamma_j$)
and show that one of the choices satisfies the required properties. Set
$$
\alpha_1 := 1_{A_i}^T M' 1_{\Gamma_j}, \alpha_2 := 1_{A_i^c}^T M' 1_{\Gamma_j}, \alpha_3 := 1_{A_i}^T M' 1_{\Gamma_j^c}, \alpha_4 := 1_{A_i^c}^T M' 1_{\Gamma_j^c}
$$
and
$$
\beta_1 := |A_i| |\Gamma_j|, \beta_2 := |A_i^c| |\Gamma_j|, \beta_3 := |A_i| |\Gamma_j^c|, \beta_4 := |A_i^c| |\Gamma_j^c|.
$$
Fix $\delta=c \lambda^{15}$ for an absolute constant $c>0$ to be chosen later.
We will show that for some $i \in \{1,2,3,4\}$, we have $|\beta_i| \ge \delta |A| |\Gamma|$ and $|\alpha_i+\rho \beta_i| \ge (\rho + \delta) \beta_i$. This implies that
if we take $A',\Gamma'$ to be the corresponding sets, then $|A'|\ge \delta |A|, |\Gamma'| \ge \delta |\Gamma|$ and
$|1_{A'} M 1_{\Gamma'}| = |\alpha_i + \rho \beta_i| \ge (\rho+\delta) |A'| |\Gamma'|$.

In order to show that, let us note that $\sum \alpha_i=0$, $|\alpha_1| \ge c_1 \lambda^5 |A| |\Gamma|$, $\beta_1 \ge c_1 \lambda^5 |A| |\Gamma|$, and the $\beta_i$ are real nonnegative numbers with $\sum \beta_i = |A| |\Gamma|$.
If for some $i$ we have $\mathrm{Re}(\alpha_i) \ge \delta |A| |\Gamma|$ then
$|\alpha_i + \rho \beta_i| \ge \mathrm{Re}(\alpha_i+\rho \beta_i) \ge \delta |A| |\Gamma|+ \rho \beta_i \ge (\rho+\delta) \beta_i$ and we are done. If $\mathrm{Re}(\alpha_i) \le -\delta |A| |\Gamma|$
then, since $\sum \alpha_i=0$, there exists some $j \ne i$ for which $\mathrm{Re}(\alpha_j) \ge \delta/3 \cdot |A| |\Gamma|$, and we are done by the previous argument. So, we may assume that $|\mathrm{Re}(\alpha_i)| \le \delta |A| |\Gamma|$
for all $i$. In particular $|\mathrm{Re}(\alpha_1)| \le (\delta/c_1 \lambda^5) \beta_1$. Hence
\begin{align*}
|\alpha_1 + \rho \beta_1|^2 & = |\rho \beta_1 + \mathrm{Re}(\alpha_1)|^2 + \mathrm{Im}(\alpha_1)^2 \\
&\ge \rho^2 \beta_1^2 + |\alpha_1|^2 - 2 \rho \beta_1 |\mathrm{Re}(\alpha_1)| \\
&\ge \beta_1^2 (\rho^2 + c_1^2 \lambda^{10} - 2 \delta/c_1 \lambda^5).\\
&\ge \beta_1^2 (\rho^2 + (c_1^2 - 2 c/c_1) \lambda^{10}),\\
\end{align*}
where we used our choice of $\delta = c \lambda^{15}$. If we choose $c>0$ small enough,  we conclude that also in this case, $|\alpha_1+\rho \beta_1| \ge (\rho + \delta) \beta_1$.
\end{proof}

We now combine Lemma~\ref{lemma:regdouble} and Lemma~\ref{lemma:minor} in order to prove Theorem~\ref{thm:spec_dbl}. The high level idea is the following.
Initialize $\rho=\eps, \Gamma=\spec{\eps}{A}$. If $M(A,\Gamma)$ is $\lambda$-regular for $\lambda=\eps \rho/150$, and $|\spec{\eps \rho/2}{A}|\approx|\Gamma|$, then the proof follows from Lemma~\ref{lemma:regdouble} and Parseval's identity. Otherwise, one of two cases must occur. The first case that could occur is that $M(A,\Gamma)$ is not $\lambda$-regular. Then by Lemma~\ref{lemma:minor} we can replace $A,\Gamma$ with $A',\Gamma'$ and increase $\rho$ by a noticeable amount. This cannot occur too many times, as $\rho \le 1$. The second case that could occur is that $|\spec{\eps \rho/2}{A}| \gg |\Gamma| \approx \spec{\rho}{A}$. In such a case, we set $\rho = \eps \rho/2$ and increase the spectrum of $A$ by a noticeable amount. As the spectrum
is bounded by $|G|$, this again cannot happen too many times. Combining these steps together requires a somewhat delicate balance act.

Let $K=K(\eps,\delta)$ be a parameter to be optimized later.
We define a sequence of sets $A_i \subseteq A$ and parameters $\rho_i \in [0,1]$ for $i \ge 1$, where initially $A_0=A, \rho_0 = \eps$. Given $A_i,\rho_i$ set $\lambda_i  = \eps \rho_i / 150$
and run the following procedure:
\begin{enumerate}
\item[(i)] If $M(A_i,\spec{\rho_i}{A_i})$ is $\lambda_i$-regular and $|\spec{\eps \rho_i/2}{A_i}| \le K |\spec{\rho_i}{A_i}|$, then set $A^*=A_i$ and finish.
\item[(ii)] If $M(A_i,\spec{\rho_i}{A_i})$ is not $\lambda_i$-regular then apply Lemma~\ref{lemma:minor} to $A_i$ and $\spec{\rho_i}{A_i}$.  Let $A' \subseteq A_i, \Gamma' \subseteq \spec{\rho_i}{A_i}$ be the resulting sets such that $|A'| \ge c \lambda_i^{15} |A_i|$, $|\Gamma'| \ge c \lambda_i^{15} |\Gamma_i|$
    and
    $|\E[M(A',\Gamma')]| \ge \rho_i + c \lambda_i^{15}$. Set $A_{i+1}=A'$
    and $\rho_{i+1} = \rho_i + (c/2) \lambda_i^{15}$. Return to step (i).
\item[(iii)] If $|\spec{\eps \rho_i/2}{A_i}| > K |\spec{\rho_i}{A_i}|$ then set $A_{i+1}=A_i$ and $\rho_{i+1}=\eps \rho_i/2$. Return to step (i).
\end{enumerate}

Next, we analyze this procedure. First, note that if the procedure ends with $A^*=A_i$ then by Lemma~\ref{lemma:regdouble} and Parseval's identity we have that
\begin{equation}\label{eq:spec_bounds_interim}
|\spec{\eps}{A^*}-\spec{\eps}{A^*}| \le 2K |\spec{\eps\rho_i/2}{A_i}| \le \frac{8 K |G|}{\eps^2 \rho_i^2 |A_i|}.
\end{equation}
So, we need to show that $\rho_i,|A_i|$ are never too small. Suppose that stages (ii) and (iii) occur $k_1$ and $k_2$ times, respectively. Let $\eta: \{ 1,\ldots,k_2\}\rightarrow \{ 1,\ldots,k_1+k_2 \}$ be the ordered indices of occurrences of stage (iii). We first bound $k_1$.

\begin{claim}\label{claim:k0}
If $i < \eta(j)$ then $\rho_i \ge (\eps/2)^j$.
\end{claim}

\begin{proof}
The value of $\rho_i$ increases in step (ii), and decreases in step (iii) by a factor of $\eps/2$. If $i < \eta(j)$ then we applied step (iii) at most $j-1$ times, hence $\rho_i \ge (\eps/2)^{j-1} \rho_0 \ge (\eps/2)^j$.
\end{proof}

\begin{claim}\label{claim:k1}
For $\forall j\in \{1,\dots,k_2-1\}$, $\left| \eta(j+1) - \eta(j)  \right| \le  (1/\eps)^{O(j)}$.
\end{claim}
\begin{proof}
Consider a step $i$ for $\eta(j) \leq i \leq \eta(j+1)$. We have that $\rho_{i+1} \ge \rho_i + (c/2) (\rho_i \eps/150)^{15} \ge \rho_i + c' \eps^{15(j+2)}$,
where $c,c'>0$ are absolute constants.  As $\rho_{i}$ never exceeds $1$ for all $i$, this process cannot repeat more than $(1/c') (1/\eps)^{15 (j+2)}$ times.
As we assume $\eps<1/2$, this is bounded by $(1/\eps)^{c' j}$ for a large enough $c'>0$.
\end{proof}

\begin{corollary}\label{cor:bound_k1}
$k_1 \le (1/\eps)^{O(k_2)}$.
\end{corollary}

\begin{proof}
By claim~\ref{claim:k1}, $k_1 \le \sum_{j=1}^{k_2} (1/\eps)^{O(j)} \le (1/\eps)^{O(k_2)}$.
\end{proof}

We next upper bound $k_2$. To do so, we will show that in step (ii) we have that $\spec{\rho_{i+1}}{A_{i+1}}$ is not much smaller than $\spec{\rho_{i}}{A_{i}}$.

\begin{claim}\label{clm:Gamma_size_step_2}
Assume that we run step (ii) in iteration $i$. Then
$$
|A_{i+1}| \ge c \lambda_i^{15}  |A_{i}|
$$
and
$$
|\spec{\rho_{i+1}}{A_{i+1}}| \ge c \lambda_i^{30} |\spec{\rho_{i}}{A_{i}}|,
$$
where $c>0$ is an absolute constant.
\end{claim}

\begin{proof}
We apply in step (ii) Lemma~\ref{lemma:minor} to $A_i, \spec{\rho_i}{A_i}$. We get subsets $A_{i+1} \subseteq A_i, \Gamma' \subseteq \spec{\rho_i}{A_i}$ such that
$|A_{i+1}| \ge c \lambda_i^{15} |A_i|$, $|\Gamma'| \ge c \lambda_i^{15} |\spec{\rho_i}{A_i}|$ and $\rho_{i+1} \leq|\E[M(A_{i+1},\Gamma')]| - (c/2) \lambda_i^{15}$. Let $S = \Gamma' \cap \spec{\rho_{i+1}}{A_{i+1}}$. Then
$$
|\E[M(A_{i+1},\Gamma')]| \le \frac{|S|}{|\Gamma'|} + \left( 1 - \frac{|S|}{|\Gamma'|}\right)\rho_{i+1}.
$$
Hence $|\spec{\rho_{i+1}}{A_{i+1}}| \ge |S| \ge (c/2) \lambda_i^{15} |\Gamma'|$ and the claim follows.
\end{proof}

Combining Claim~\ref{claim:k1} and Claim~\ref{clm:Gamma_size_step_2}, we deduce that, for any $j \in \{1,\ldots,k_2-1\}$, the ratio in the size of the spectrums
immediately after the $j$-th application of step (iii), and immediately before the $j+1$ application of step (iii), is lower bounded by
\begin{align*}
T_j &:= \frac{|\spec{\rho_{\eta(j)}}{A_{\eta(j)}}|}{|\spec{\rho_{\eta(j+1)-1}}{A_{\eta(j+1)-1}}|} \le \prod_{i=\eta(j)}^{\eta(j+1)-2} \frac{1}{c \lambda_i^{30}} \le \left(\frac{1}{c}  \left(\frac{150 \cdot 2^j}{\eps^{j+1}}\right)^{30} \right)^{\eta(j+1)-\eta(j)} \\
&\le (1/\eps)^{O(j \cdot (1/\eps)^{O(j)})} \le \exp\left( (1/\eps)^{O(j)}\right).
\end{align*}
We will choose $K$ large enough so that $T_j \le K^{1/2}$ for all $j < k_2$, and hence
$$
|\spec{\rho_{\eta(j+1)}}{A_{\eta(j+1)}}| \ge K \cdot |\spec{\rho_{\eta(j+1)-1}}{A_{\eta(j+1)-1}}| \ge K^{1/2} \cdot |\spec{\rho_{\eta(j)}}{A_{\eta(j)}}|.
$$
Fix $K=|G|^{\delta}$ and $C=\exp((1/\eps)^{O(1/\delta)})$. We may assume that $|G| \ge C$, as otherwise our bounds are trivial. Then,
we must have $k_2 \le 2/\delta$ and hence $k_1 \le (1/\eps)^{O(1/\delta)}$. We conclude that
$$
\frac{|A|}{|A^*|} \le \prod_{i=1}^{k_1+k_2} \frac{1}{c \lambda_i^{15}} \le \exp\left((1/\eps)^{O(1/\delta)}\right)
$$
and that plugging these estimates into Equation \eqref{eq:spec_bounds_interim} implies that
$$
\left| \spec{\eps}{A^*} - \spec{\eps}{A^*}\right| \le \exp \left((1/\eps)^{O(1/\delta)} \right) \cdot  |G|^{1-\alpha+\delta}.
$$
Since the definition of the spectrum is symmetric, $\spec{\eps}{A^*}=-\spec{\eps}{A^*}$, this implies the same bounds on $\left| \spec{\eps}{A^*} + \spec{\eps}{A^*}\right|$.

\section{Proof of Theorem~\ref{thm:spec_dbl2}}
\label{sec:proof2}

The proof of theorem~\ref{thm:spec_dbl2} is very similar to the proof of theorem~\ref{thm:spec_dbl}, with a few small tweaks.
First, we use Lemma~\ref{lemma:regular} and Lemma~\ref{lemma:regdouble} in the special case of $\rho = \eps$. We restate Lemma~\ref{lemma:regdouble} in this special case.
\begin{lemma}\label{lemma:regdouble2}
If $M = M(A, \spec{\eps}{A}) $ is $\frac{\eps^2}{150}$-regular, then
$$\left| \spec{\eps}{A} - \spec{\eps}{A}\right| \leq 2 \frac{\left| \spec{\eps^2/2}{A}\right|^2}{ \left| \spec{\eps}{A}\right|} .$$
\end{lemma}

We combine Lemma~\ref{lemma:regdouble2} with Lemma~\ref{lemma:minor} to prove Theorem~\ref{thm:spec_dbl2}.
The difference is in the iterative refinement process. Here, instead of setting $\lambda_i = \eps\rho_i/150$, we instead set $\lambda_i = \rho_i^2/150$. To be more precise, initialize $\Gamma = \spec{\eps}{A}$. If $M(A,\Gamma)$ is $\lambda$-regular for $\lambda = \eps^2/150$, and $|\spec{\eps^2/2}{A}|\approx |\Gamma|$, then the proof follows from Lemma~\ref{lemma:regdouble2} and Parseval's identity. Otherwise, one of the following two cases must occur. The first case that could occur is that $M(A,\Gamma)$ is not $\lambda$-regular. In this case, by Lemma~\ref{lemma:minor} we can replace $A$, $\Gamma$ with $A^\prime$,$\Gamma^\prime$ and increase $\eps$ by a noticeable amount. This can not occur many times as $\eps\leq 1$. The other case that can occur is that $|\spec{\eps^2/2}{A}|\gg |\Gamma| \approx \spec{\eps}{A}$. In this case, we set $\eps = \eps^2/2$ and increase the spectrum of $A$. Since the spectrum is bounded by $|G|$, this also can not occur too many times. In the following we formalize this high level argument.

Let $K=K(\eps,\delta)$ be a parameter to be optimized later.
Define a sequence of sets $A_i\subseteq A$ and parameters $\rho_i\in [0,1]$ for $i\geq 1$, and initialize $A_0 = A$ and $\rho_0 = \eps$. Recall that $\delta$ is a parameter, chosen so that the final doubling constant
is bounded by $|G|^{\delta}$. Given $A_i,\rho_i$ set $\lambda_i = \rho_i^2/150$ and run the following procedure:
\begin{enumerate}
\item[(i)] If $M(A_i,\spec{\rho_i}{A_i})$ is $\lambda_i$-regular and $|\spec{\rho_i^2/2}{A_i}| \le K |\spec{\rho_i}{A_i}|$, then set $A^*=A_i$ and finish.
\item[(ii)] If $M(A_i,\spec{\rho_i}{A_i})$ is not $\lambda_i$-regular then apply Lemma~\ref{lemma:minor} to $A_i,\spec{\rho_i}{A_i}$.  Let $A' \subseteq A_i, \Gamma' \subseteq \spec{\rho_i}{A_i}$ be the resulting sets such that $|A'| \ge c \lambda_i^{15} |A_i|$, $|\Gamma'| \ge c \lambda_i^{15} |\Gamma_i|$ and
    $|\E[M(A',\Gamma')]| \ge \rho_i + c \lambda_i^{15}$. Set $A_{i+1}=A'$
    and $\rho_{i+1} = \rho_i + (c/2) \lambda_i^{15}$.
\item[(iii)] If $|\spec{ \rho_i^2/2}{A_i}| > K  |\spec{\rho_i}{A_i}|$ then set $A_{i+1}=A_i$ and $\rho_{i+1}= \rho_i^2/2$.
\end{enumerate}

The analysis of this procedure is similar to the analysis of the procedure in the proof of Theorem~\ref{thm:spec_dbl}. First note that if the procedure ends with $A^* = A_i$ and $\eps^* = \rho_i$ then by Lemma~\ref{lemma:regdouble2} we have that
\begin{equation}\label{eq:spec_bounds_interim2}
|\spec{\eps^*}{A^*}-\spec{\eps^*}{A^*}|\leq 2K|\spec{{\eps^*}^2/2}{A^*}|\leq 2K^2 |\spec{\eps^*}{A^*}|.
\end{equation}

Therefore, we need to show that $\eps^*$ and $|A^*|$ are not too small. Suppose that stages (ii) and (iii) occur $k_1$ and $k_2$ times, respectively. Let $\eta:\{1,\cdots,k_2\}\rightarrow \{1,\cdots,k_1+k_2\}$ be the ordered indices of occurrences of stage (iii). We first bound $k_1$.

\begin{claim}\label{claim:k0'}
If $i<\eta(j)$ then $\rho_i\geq (\eps/2)^{2^{j}}$.
\end{claim}
\begin{proof}
The value of $\rho_i$ increases in step (ii), and decreases in step (iii). If $i < \eta(j)$ then we applied step (iii) at most $j-1$ times, hence $\rho_i \ge (\eps/2)^{2^{j}}$.
\end{proof}

\begin{claim}\label{claim:k1'}
For $\forall j\in \{1,\dots,k_2-1\}$, $\left| \eta(j+1) - \eta(j)  \right| \le (1/\eps)^{O(2^j)}$.
\end{claim}
\begin{proof}
Consider a step $i$ for $\eta(j) \leq i \leq \eta(j+1)$. We have that $\rho_{i+1} \ge \rho_i + c (\rho_i^2)^{15} \ge \rho_i + c ((\eps/2)^{30 \cdot 2^j})$.
As $\rho_{i}$ never exceeds $1$ for all $i$, this process cannot repeat more than $(1/c) (2/\eps)^{30 \cdot 2^j}$ times.
\end{proof}

\begin{corollary}\label{cor:bound_k1'}
$k_1 \le (1/\eps)^{O(2^{k_2})}$.
\end{corollary}

\begin{proof}
By claim~\ref{claim:k1'}, $k_1 \le \sum_{j=1}^{k_2} (1/\eps)^{O(2^j)} \le (1/\eps)^{O(2^{k_2})}$.
\end{proof}

We next upper bound $k_2$. To do so, we will show that in step (ii) we have that $\spec{\rho_{i+1}}{A_{i+1}}$ is not much smaller than $\spec{\rho_{i}}{A_{i}}$.
We restate Claim~\ref{clm:Gamma_size_step_2} which was proved before.
\begin{claim}\label{clm:Gamma_size_step_2'}
Assume that we run step (ii) in iteration $i$. Then
$$
|A_{i+1}| \ge c \lambda_i^{15}  \cdot |A_{i}|
$$
and
$$
|\spec{\rho_{i+1}}{A_{i+1}}| \ge c \lambda_i^{30}  \cdot |\spec{\rho_{i}}{A_{i}}|.
$$
\end{claim}

As in the proof of Theorem~\ref{thm:spec_dbl}, if we combine Claim~\ref{claim:k1'} and Claim~\ref{clm:Gamma_size_step_2'},
then for any $j \in \{1,\ldots,k_2-1\}$, the ratio in the size of the spectrums
immediately after the $j$-th application of step (iii), and immediately before the $j+1$ application of step (iii), is lower bounded by
\begin{align*}
T_j &:= \frac{|\spec{\rho_{\eta(j)}}{A_{\eta(j)}}|}{|\spec{\rho_{\eta(j+1)-1}}{A_{\eta(j+1)-1}}|} \le \exp\left((1/\eps)^{O(2^{j})} \right).
\end{align*}
We will choose $K$ large enough so that $T_j \le K^{1/2}$ for all $j < k_2$, and hence
$$
|\spec{\rho_{\eta(j+1)}}{A_{\eta(j+1)}}| \ge K \cdot |\spec{\rho_{\eta(j+1)-1}}{A_{\eta(j+1)-1}}| \ge K^{1/2} \cdot |\spec{\rho_{\eta(j)}}{A_{\eta(j)}}|.
$$
Fix $K=|G|^{\delta/2}$ and $C=\exp((1/\eps)^{O(2^{4/\delta})})$. We may assume that $|G| \ge C$, as otherwise our bounds are trivial.
Then we deduce that $k_2 \le 4/\delta$, $k_1 \le (2/\eps)^{O(2^{4/\delta})}$. We get that
$$
\frac{|A|}{|A^*|} \le \prod_{i=1}^{k_1+k_2} \frac{1}{c  (\lambda_i)^{15}} = \exp\left((1/\eps)^{O(2^{4/\delta})}\right)
$$
and then by plugging these estimates into Equation~\eqref{eq:spec_bounds_interim2} we conclude that
$$
\left| \spec{\eps^*}{A^*} - \spec{\eps^*}{A^*}\right| \le \exp\left((1/\eps)^{O(2^{4/\delta})}\right) \left|G\right|^\delta \cdot \left|\spec{\eps^*}{A^*}\right|.
$$
Since the definition of the spectrum is symmetric, $\spec{\eps^*}{A^*}=-\spec{\eps^*}{A^*}$, this implies the same bounds on $\left| \spec{\eps^*}{A^*} + \spec{\eps^*}{A^*}\right|$.

\bibliographystyle{unsrt}
\bibliography{fuzzy}

\end{document}